\documentclass[3p,times,12pt]{elsarticle}

\usepackage{amsmath,amssymb,amsthm,enumerate}
\usepackage{calrsfs} %mathcal の表記を変える．フランス語の筆記体になる．
\usepackage{braket}%ブラケットの表示<x>
\theoremstyle{plain}
\newtheorem{thm}{Theorem}[section]
\newtheorem{prop}[thm]{Proposition}
\newtheorem{lem}[thm]{Lemma}
\newdefinition{rem}[thm]{Remark}
\newdefinition{defi}[thm]{Definition}
\def\Rnum#1{\uppercase\expandafter{\romannumeral #1}}%roman numeral
\def\coloneqq{\mathrel{\mathop:}=}%
\makeatletter
    
    \@addtoreset{equation}{section}
\makeatother
\journal{}

\begin{document}

\begin{frontmatter}

%% Title, authors and addresses

%% use the tnoteref command within \title for footnotes;
%% use the tnotetext command for theassociated footnote;
%% use the fnref command within \author or \address for footnotes;
%% use the fntext command for theassociated footnote;
%% use the corref command within \author for corresponding author footnotes;
%% use the cortext command for theassociated footnote;
%% use the ead command for the email address,
%% and the form \ead[url] for the home page:
%% \title{Title\tnoteref{label1}}
%% \tnotetext[label1]{}
%% \author{Name\corref{cor1}\fnref{label2}}
%% \ead{email address}
%% \ead[url]{home page}
%% \fntext[label2]{}
%% \cortext[cor1]{}
%% \address{Address\fnref{label3}}
%% \fntext[label3]{}

\title{Well-posedness for a generalized derivative nonlinear Schr\"{o}dinger equation}

%% use optional labels to link authors explicitly to addresses:
%% \author[label1,label2]{}
%% \address[label1]{}
%% \address[label2]{}

\author{MASAYUKI HAYASHI AND TOHRU OZAWA}
% \author[1]{MASAYUKI HAYASHI}
% \address[1]{Department of Applied Physics, Waseda University, Tokyo 169-8555, Japan}
% \author[2]{TOHRU OZAWA}
% \address[2]{Department of Applied Physics, Waseda University, Tokyo 169-8555, Japan}

\begin{abstract}
\par We study the Cauchy problem for a generalized derivative nonlinear Schr\"{o}dinger equation with the Dirichlet boundary condition. We establish the local well-posedness results in the Sobolev spaces $H^1$ and $H^2$. Solutions are constructed as a limit of approximate solutions by a method independent 
of a compactness argument. We also discuss the global existence of solutions in the energy space $H^1$.
\end{abstract}

\begin{keyword}
Derivative nonlinear Schr\"{o}dinger equation; Yosida regularization
%% keywords here, in the form: keyword \sep keyword

%% PACS codes here, in the form: \PACS code \sep code

%% MSC codes here, in the form: \MSC code \sep code
%% or \MSC[2008] code \sep code (2000 is the default)

\end{keyword}

\end{frontmatter}

%% \linenumbers

%% main text
\section{Introduction}
 We consider the Cauchy problem for the following generalized derivative nonlinear Schr\"{o}dinger equation (gDNLS) with the Dirichlet boundary condition
 \begin{align}
\begin{cases}
 i \partial_{t}u + \partial_{x}^2 u +i|u|^{2\sigma}\partial_{x}u= 0, \ (t,x) \in 
\mathbb{R} \times \Omega , \\ 
u(t,x)=0,  \ (t,x) \in  \mathbb{R} \times \partial \Omega ,\\
u(0,x) =\varphi (x) , \ x \in \Omega ,
\end{cases}
\label{NLS}
\end{align}
where $u$ is a complex valued function of $(t,x)\in \mathbb{R} \times \Omega$, $\sigma >0$ and $\Omega \subset \mathbb{R}$ is an open interval.  With $\sigma =1$, (\ref{NLS}) has appeared as a model for ultrashort optical pulses \cite{Moses}. The solution of (\ref{NLS}) obeys formally the following charge and energy conservation laws:
\begin{align}
&M(u(t))\coloneqq \int_{\Omega} |u|^2 dx = M(\varphi ), \\
&E(u(t)) \coloneqq \int_{\Omega} \Bigl( |\partial_{x}u|^2 +
\frac{1}{\sigma +1}\mathrm{Im} |u|^{2\sigma}\overline{u}\partial_{x}u \Bigr) dx =E(  \varphi ).
\end{align}
When $\sigma =1$ and $\Omega =\mathbb{R}$, if $u$ is a solution of (\ref{NLS}), the gauge transformed solution $v$ defined by
\begin{align*}
v(t,x)=u(t,x)\exp \left( -\frac{i}{2} \int_{-\infty}^{x} |u(t,y)|^2dy \right) ,
\end{align*}
satisfies the standard derivative nonlinear Schr\"{o}dinger equation (DNLS):
\begin{align}
i\partial_t v +\partial_{x}^2 v +i\partial_x (|v|^2v)= 0, \ (t,x) \in 
\mathbb{R} \times \mathbb{R}.
\label{dNLS}
\end{align}
The DNLS equation appears in plasma physics as a model for the propagation of Alfv\'{e}n waves in magnetized plasma (see \cite{Mio}, \cite{Sulem} ). The Cauchy problem for (\ref{dNLS}) has been studied by many authors. The local and global well-posedness in the Sobolev spaces $H^s$ with $s\geq 1$ is studied in \cite{Tsutsumi}, \cite{Hayashi}, 
\cite{HO1}, \cite{HO2}, \cite{HO3}. Solutions of low regularity have been studied in \cite{Takaoka}, \cite{CKSTT}, \cite{CKSTT2}, \cite{Hr}, \cite{GH}. The DNLS equation in a bounded domain $\Omega =(a,b)$ with zero Dirichlet boundary condition is studied in \cite{Chen}, \cite{Tan}.
\par There are only a few results for the equation (\ref{NLS}) with general exponents $\sigma >0$, as compared with $\sigma =1$. Hao \cite{Hao} proved local well-posedness in $H^{1/2}(\mathbb{R})$ intersected with an appropriate Strichartz space for $\sigma \geq 5/2$ by using the gauge transformation and the Littlewood-Paley decomposition. Liu-Simpson-Sulem \cite{LSS} studied the orbital stability and instability of solitary waves for (\ref{NLS}) depending on the value of $\sigma$. We should note that in \cite{LSS} the existence of $H^1$ solution for $\sigma >0$ with the initial data $\varphi \in H^1(\mathbb{R})$ is assumed. Ambrose-Simpson \cite{Ambrose} proved the existence and uniqueness of solutions $u \in C([0,T];H^2(\mathbb{T}))$ and the existence of solution $u\in L^{\infty}((0,T);H^1(\mathbb{T}))$ for $\sigma \geq 1$. The construction of solutions depends on a compactness argument and the uniqueness of $H^1$-solutions is not proved. Recently, Santos \cite{Santos} proved the existence and uniqueness of 
solutions $u \in L^{\infty}((0,T); H^{3/2}(\mathbb{R})\cap \braket{x}^{-1}H^{1/2}(\mathbb{R}))$ for sufficient small initial data in the case of $1/2 <\sigma <1$. The proof of \cite{Santos} is based on parabolic regularization and smoothing properties associated with the Schr\"{o}dinger group, where the weighted Sobolev space is essential to control the mixed norm $L^p_xL^q_t$. He also proved the existence and uniqueness of solutions $u\in C([0,T] ;H^{1/2}(\mathbb{R}))$ for sufficient small initial data in the case of $\sigma >1$.
\par The aim of this paper is to construct $H^1$ and $H^2$-solutions of (\ref{NLS}) for $\sigma \geq 1/2$. In the case of $1/2 \leq \sigma <1$, the nonlinear term $|u|^{2\sigma}$ is not even $C^2$, and therefore a delicate argument is needed. Our first main result is on the local well-posedness in $H^2$ for $\sigma \geq 1/2$.
\begin{thm}
\label{H2lwp}
Let $\sigma \geq 1/2$. Let $\varphi \in H^2(\Omega ) \cap 
H^1_0(\Omega )$. Then there exists $T>0$ and a unique solution $u \in C([-T,T]; H^2(\Omega ) \cap H^1_0(\Omega ))$ of {\rm (\ref{NLS})}. Moreover, $u$ depends continuously on 
$\varphi$ in the following sense. If $\varphi_n \rightarrow \varphi \ in \ H^2(\Omega )$ as $n \rightarrow \infty$ and if 
$u_n$ is the corresponding solution of {\rm (\ref{NLS})}, then $u_n$ is defined on 
the same interval $[-T,T]$ for n large enough and $u_n \rightarrow u \ \mathrm{in} \ C([-T,T]; H^s(\Omega ))$ as $n\rightarrow \infty$ for all $0\leq s <2$.
\end{thm}
%%%%%%%%%%%%%%%%%%%%
\begin{rem}
When $\sigma =1/2$, the nonlinear term $i|u|\partial_xu$ is quadratic. Christ 
\cite{Christ} considered the following Cauchy problem: 
 \begin{align}
\begin{cases}
 i \partial_{t}u + \partial_{x}^2 u +iu\partial_x u= 0, \ t>0, x \in \mathbb{R} ,\\ 
u(0,x) =\varphi (x) , \ x \in \mathbb{R},
\end{cases}
\label{Chr}
\end{align}
he proved the norm inflation in any Sobolev space $H^s(\mathbb{R})$ with $ s \in \mathbb{R}$ (i.e. $\|u(0) \|_{H^s} \ll 1$ but $\|u(t)\|_{H^s} \gg 1$ for some $t \ll 1$). Theorem \ref{H2lwp} tells us that the behavior of the solution of (\ref{NLS}) is very different from that of the solution of (\ref{Chr}) even though both equations have the quadratic nonlinear term with derivative.
\end{rem}
%%%%%%%%%%%%%%%%%%%%
\par The proof of Theorem \ref{H2lwp} proceeds in four steps. We first employ a Yosida-type regularization and construct  approximate solutions. Next, we follow an argument in \cite{Ambrose} and obtain the uniform estimates on the approximate solutions in $H^1$ by using the conservation laws. Under the uniform bounds in $H^1$, we obtain the uniform estimates in $H^2$ by estimating time derivative of approximate solutions. More precisely, we differentiate the equation once in time instead of differentiating twice the equation in space in order to obtain $H^2$ estimates. This enables us to relax the smoothness condition of the nonlinear term. This idea is from Kato \cite{Kato}. Finally, we prove the sequence of approximate solutions is a Cauchy sequence in $L^2$ and construct the solution of (\ref{NLS}) by the completeness of a function space. We remark that the argument of constructing solutions does not need any compactness theorem, for example, the Ascoli-Arzel\`{a} theorem, the Rellich-Kondrachov theorem, the Banach-Alaoglu theorem, etc.
%%%%%%%%%%%%%%%%%%%%%%%%%%%%%%%%%%%%%%%
%%%%%%%%%%%%%%%%%%%%%%%%%%%%%%%%%%%%%%%
\par Santos \cite{Santos} proved the uniqueness in $L^{\infty}((0,T);  H^{3/2}(\mathbb{R})\cap \braket{x}^{-1}H^{1/2}(\mathbb{R})))$ for $1/2 <\sigma <1$. We found that it is not necessary to use the weighted Sobolev space for the uniqueness.  
\begin{thm}
\label{H3/2}
Let $\sigma \geq 1/2$. Let $\varphi \in H^{3/2}(\Omega )\cap H^1_0(\Omega )$ and $T>0$. If $u$ and $v$ are two solutions of (\ref{NLS}) in $L^{\infty}((-T,T);H^{3/2}(\Omega )\cap H^1_0(\Omega ))$ with the same initial data, then $u=v$.
\end{thm}
Our proof of Theorem \ref{H3/2} is based on Yudovitch type argument \cite{Yud}. Related proofs for the nonlinear Schr\"{odinger} equation are given in \cite{V}, \cite{Ogawa}, \cite{OO}.
\par Next result is on the local well-posedness in $H^1$ for $\sigma \geq 1$.
%%%%%%%%%%%%%%%%%%%%
%%%%%%%%%%%%%%%%%%%%
\begin{thm}
\label{H1lwp}
Let $\sigma \geq 1$ and let $\Omega$ be an unbounded interval. Let $\varphi \in H^1_0(\Omega )$. Then there exists $T>0$ and a unique solution $u \in C([-T,T]; H^1_0(\Omega )) \cap L^4( (-T,T);W^{1,\infty }(\Omega ))$ of {\rm (\ref{NLS})}. Moreover, the following property holds:
\begin{enumerate}[(i)]
\item $u \in L^q((-T,T); W^{1,r}(\Omega ))$ for every admissible pair 
$(q,r)$ i.e. $\ 0 \leq 2/q=1/2-1/r \leq 1/2$.
\item $M(u(t))=M(\varphi )$,  $E(u(t))=E(\varphi )$ for all $t \in [-T,T]$. 
\item $u$ depends continuously on $\varphi$ in the following sense. If $\varphi_n \rightarrow \varphi \ in \ H^1_0(\Omega )$ as $n \rightarrow \infty$ and if 
$u_n$ is the corresponding solution of {\rm (\ref{NLS})}, then $u_n$ is defined on 
the same interval $[-T,T]$ for n large enough and $u_n \rightarrow u \ \mathrm{in} \ C([-T,T]; H^1_0(\Omega ))$.
\end{enumerate}
\end{thm}
%%%%%%%%%%%%%%%%%%%
%%%%%%%%%%%%%%%%%%%
Our proof of Theorem \ref{H1lwp} depends on the gauge transformation and the Strichartz estimate. Since the Strichartz estimate does not hold in a bounded domain, we need to assume $\Omega$ is an unbounded interval. We employ $H^2$-solutions constructed in Theorem \ref{H2lwp} as approximate solutions.  Firstly, we derive the differential equation by using the gauge transformation that the spatial derivative of approximate solution should satisfy. Next, we obtain the uniform estimates on the approximate solutions in $L^q_tW^{1,r}_x$ for any admissible pair $(q,r)$ by using the Strichartz estimate. Finally, we prove the sequence of approximate solutions is a Cauchy sequence in $L^2$ and construct 
the $H^1$-solution of (\ref{NLS}). The last step is similar to that of the proof of Theorem \ref{H2lwp}. This method is required that the nonlinear term is of $C^2$, so we need to assume $\sigma \geq 1$. 
\par From the conservation of energy and  iterating Theorem \ref{H1lwp}, we can prove the global well-posedness in $H^1$.
\begin{thm}
\label{H1gwp}
Let $\sigma \geq 1$ and let $\Omega$ be an unbounded interval. Then there 
exist $\varepsilon_0, \ \varepsilon_1>0$ such that if $\varphi \in H^1_0(\Omega )$ satisfies
\begin{align*}
&\| \varphi \|_{L^2} \leq \varepsilon_0 \quad when \ \sigma =1, \\
&\| \varphi \|_{H^1}  \leq \varepsilon_1 \quad when\  \sigma >1 ,
\end{align*}
then there exists a unique solution $u \in C(\mathbb{R}; H^1_0(\Omega )) \cap L^4_{\mathrm{loc}}( \mathbb{R};W^{1,\infty }(\Omega ))$ of {\rm (\ref{NLS})}. Moreover, the following property holds:
\begin{enumerate}[(i)]
\item $u \in L^q_{\mathrm{loc}}(\mathbb{R}; W^{1,r}(\Omega ))$ for every admissible pair 
$(q,r)$.
\item $M(u(t))=M(\varphi )$,  $E(u(t))=E(\varphi )$ for all $t \in \mathbb{R}$ . 
\item $u$ depends continuously on $\varphi$ in the following sense. If $\varphi_n \rightarrow \varphi \ in \ H^1_0(\Omega )$ as $n \rightarrow \infty$ and if 
$u_n$ is the corresponding solution of {\rm (\ref{NLS})}, then $u_n \rightarrow u \ \mathrm{in} \ C([-T,T]; H^1_0(\Omega ))$ for all $T>0$.
\end{enumerate}
\end{thm}
%%%%%%%%%%%%%%%%%%%
\begin{rem}
In the case of $\sigma =1,\ \Omega =\mathbb{R}$, Wu \cite{Wu} proved that if $\| \varphi \|_{L^2}<\sqrt{4\pi}$, Theorem \ref{H1gwp} follows by using sharp Gagliardo-Nirenberg inequality and the momentum conservation law
\begin{align*}
P(u(t))\coloneqq \mathrm{Im} \int_{\mathbb{R}} \overline{u}\partial_{x}u \ dx =P(\varphi ).
\end{align*}
\end{rem}
%%%%%%%%%%%%%%%%%%%
%%%%%%%%%%%%%%%%%%%
\par In the case of $\thinspace \sigma <1$, we obtain the following result.
\begin{thm}
\label{H1global}
Let $\thinspace 0< \sigma <1$. Let $\varphi \in H^1_0(\Omega )$. Then there exists a solution $u \in (C_w \cap L^{\infty})(\mathbb{R}; H^1_0(\Omega ))$ of {\rm (\ref{NLS})}. In 
addition, 
\begin{align*}
M(u(t))=M(\varphi ) \quad and \quad E(u(t))\leq E(\varphi )
\end{align*}
for all $t \in \mathbb{R}$ .
\end{thm}
\begin{rem}
When $0< \sigma <1$, we do not need to assume the smallness of the initial data for the global existence of the solution. Since the solution is constructed by a compactness argument, we do not know whether the solutions given in Theorem \ref{H1global} is unique or not. If uniqueness holds in $L^{\infty}(\mathbb{R};H^1_0(\Omega ))$, we can prove 
easily that $E(u(t))=E(\varphi )$ for all $t \in \mathbb{R}$ and that $u \in C(\mathbb{R}; H^1_0(\Omega ))$. 
\end{rem}
\par The paper is organized as follows. Section \ref{SH2} is concerned with local well-posedness in $H^2$. Theorem {\ref{H2lwp}} will be proved in Section \ref{SH2}. Theorem {\ref{H3/2}} will be proved in Section \ref{SH3/2}. Well-posedness in $H^1$ is considered in Section \ref{SH1}. Theorem \ref{H1lwp} and Theorem \ref{H1gwp} will be proved in Section \ref{SH1}. Finally, we prove Theorem \ref{H1global} in Section \ref{SH12}.\\[10pt]
%%%%%%%%%%%
%notations
%%%%%%%%%%%
% \par Throughout the paper we use the following nontations.\\
{\bf Notation.}\quad $C^{\infty}_c (\Omega )$ denotes the space of complex-valued $C^{\infty}$-functions on $\Omega$ with compact support in $\Omega$. For any $p$ with $1\leq p\leq \infty$, $L^p(\Omega )$ denotes the usual Lebesgue space and $p'$ is the dual exponent defined by $1/p+1/p'=1$. The usual scalar product on $L^2(\Omega)$ is denoted by $(\cdot ,\cdot )$. For any $p$ with $1\leq p\leq \infty$ and any $m\in \mathbb{N}$, $W^{m,p}(\Omega )$ denotes the usual Sobolev space of order $m$. If $p=2$, $W^{m,p}(\Omega )$ is also written as $H^m$. If $s>0$ is not an integer, $H^s(\Omega )= \{ f\in L^2(\Omega) ; \| f\|_{H^s(\Omega )}< \infty \}$ with
 \begin{align*}
 \| f\|_{H^s(\Omega )}^2=\| f\|_{H^m(\Omega )}^2+\| \partial_x^m f\|_{H^{r}(\Omega )}^2,
 \end{align*}
 where $m$ is an non-negative integer and $0<r<1$ such that $s=m+r$, and
 \begin{align*}
 \| f\|_{H^r(\Omega )}^2 =\int_{\Omega}\int_{\Omega} \frac{|f(x)-f(y)|^2}{|x-y|^{1+2r}} dxdy .
 \end{align*}
 For $m \in \mathbb{N}$, $H^m_0(\Omega )$ denotes the completion of $C^{\infty}_c(\Omega )$ in $H^m(\Omega )$, and $H^{-m}(\Omega )$ denotes the dual of $H^m_0(\Omega )$. For any interval $I \subset \mathbb{R}$ and any Banach space 
$X$, we denote by $C(I;X)$ (resp. $C_w(I;X)$) the space of strongly (resp. weakly) continuous functions from $I$ to $X$. $L^p(I;X)$ denotes the usual Bochner space and $W^{m,p}(I;X)$ denotes the usual vector-valued Sobolev space. If $G:X\rightarrow \mathbb{R}$ is G\^{a}teaux differentiable and
\begin{align*}
G'(u)v &= \lim_{t \rightarrow 0}\frac{G(u+tv)-G(u)}{t} \\
&=2\mathrm{Re} \bigl( g(u) , v \bigr) 
\end{align*}
for all $u, v \in X$, we denote by $G'(u)=g(u)$. $U(t)=\exp (it\partial_x^2 )$ denotes the free propagator of Schr\"{o}dinger equation. A different positive constant might denoted by the same letter $C$. If necessary, we denote by $C(\ast ,...,\ast )$ constants depending on the quantities appearing in parentheses.

%%%%%%%%%%%%%%%%%%%%%%%%%%%%
%%%%%%%%%%%%%%%%%%%%%%%%%%%%
\section{Well-posedness in $H^2$}
\label{SH2}
\subsection{Approximate solutions}
Let $g(u)$ and $G(u)$ be defined by 
\begin{align*}
g(u)&=i|u|^{2\sigma}\partial_x u ,\\
G(u)&=\frac{1}{\sigma +1} \int_{\Omega}\mathrm{Im} |u|^{2\sigma}\overline{u}\partial_{x}u  dx
\end{align*}
for $\sigma >0$. Then we see that 
\begin{align*}
G \in C^1 (H^1_0(\Omega ); \mathbb{R}),\ G'=g.
\end{align*}
For any nonnegative integer $m$, we consider the following approximate problem:
 \begin{align}
 \begin{cases}
 i \partial_{t}u_m + \partial_{x}^2 u_m +J_mg(J_mu_m)= 0, \\
 u_m(0)=\varphi ,
 \end{cases}
 \label{NLS2}
 \end{align}
where $J_m$ is Yosida type approximation defined by
\begin{align}
\label{Jm}
J_m=\left( I-\frac{1}{m}\partial_x^2 \right)^{-1}.
\end{align}
Note that $\partial_x^2$ is self-adjoint in $L^2(\Omega )$ with domain $H^2(\Omega ) \cap H^1_0(\Omega )$. We recall the following main properties of $J_m$. For the proof one can see \cite{Cazenave}.
\begin{prop}
\label{Jm1}
Let $X$ be any of the spaces $H^2(\Omega),\ H^1_0(\Omega),\ H^{-1}(\Omega),\ $and $L^p(\Omega)$ with $1<p<\infty$ and let $X^{\ast}$ be its dual space.  Then:
\begin{enumerate}[(i)]
\item $\braket{J_mf,g}_{X,X^{\ast}}=\braket{f,J_mg}_{X,X^{\ast}},\ \forall f\in X, \ \forall g\in X^{\ast}$.
% \item $\| J_m \|_{\mathcal{L}(L^2,H^2)}\leq cm, \ c>0$.
\item $J_m\in \mathcal{L}(L^2;H^2)$. 
\item $\| J_m \|_{\mathcal{L}(X,X)} \leq 1$. 
\item  $J_mu \rightarrow u \ in\  X \ (m\rightarrow \infty ), \ \forall u \in X$.
\item $\sup_{m \in \mathbb{N}}\| u_m\|_{X} <\infty ,\ \Rightarrow J_mu_m -u_m \rightharpoonup 0\ in \ X\ (m\rightarrow \infty )$. 
\label{Jmweak}
\end{enumerate}
\end{prop}
If $\varphi \in H^2(\Omega )\cap H^1_0(\Omega )$ is given, by Proposition \ref{Jm1} and the Banach fixed-point theorem, there exists $T_m>0$ and 
$u_m \in C([-T_m, T_m]; H^2(\Omega)\cap H^1_0(\Omega))$ which is a solution of the initial value problem (\ref{NLS2}). \\
\indent Next, we establish the uniform bounds on the solutions in $H^2$ with respect to $m$. This will allow us to construct a solution of (\ref{NLS}) in the limit $m \rightarrow \infty$. We define 
\begin{align*}
g_m(u)=J_m(g(J_mu)) \quad \mathrm{and} \quad G_m(u)=G(J_mu).
\end{align*}
Then we see that
\begin{align*}
G_m \in C^1 (H^1_0(\Omega ); \mathbb{R}),\ G'_m=g_m.
\end{align*}
We introduce an approximate energy:
\begin{align}
E_m(u)=\int_{\Omega} |\partial_{x}u|^2 dx+G_m(u).
\label{menergy}
\end{align}
A standard calculation shows the conservation of charge and energy for the approximate problem.
\begin{lem}
\label{H1lem}
For all $t \in [-T_m,T_m]$ ,
\begin{enumerate}[(i)]
\item $M(u_m(t))=M(\varphi )$, 
\label{1}
\item $E_m(u_m(t))=E_m(\varphi )$.
\label{2}
\end{enumerate}
\end{lem}
We need the following lemma to obtain the uniform $H^1$ estimates of $(u_m)_{m \in \mathbb{N}}$.
\begin{lem}
\label{H1lem1}
For any $r \geq 1$ there exists $C>0$ such that
\begin{align*}
\frac{d}{dt} \int_{\Omega} |u_m|^{2r} dx \leq C\bigl( 1+\| u_m\|^2_{H^1} \bigr)^{r+\sigma} ,
\end{align*}
where the positive constant $C$ is independent of $m$. 
\end{lem}
\begin{proof}
The lemma follows from a direct calculation as
\begin{align*}
\frac{d}{dt} \int_{\Omega} |u_m|^{2r} dx &= \int_{\Omega}2r |u_m|^{2(r-1)}\mathrm{Re}(\partial_t u_m \overline{u_m}) \\
&= \int_{\Omega} 2r|u_m|^{2(r-1)} \mathrm{Im}\biggl(  (-\partial_x^2 u-g_m(u_m))\overline{u_m}\biggr) \\
&= \int_{\Omega}2r \mathrm{Im} \biggl( \partial_x u_m \partial_x (|u_m|^{2(r-1)} \overline{u_m})-|u_m|^{2(r-1)} g_m(u_m)\overline{u_m}\biggr) \\[3pt]
&\leq C\bigl( \| u_m\|_{L^{\infty}}^{2(r-1)}\| \partial_x u_m\|_{L^2}^2+\| u_m\|_{L^\infty}^{2(r+\sigma -1)}\| \partial_x u_m\|_{L^2}\| u_m\|_{L^2}\bigr) \\[3pt]
&\leq C\bigl( 1+\| u_m\|^2_{H^1} \bigr)^{r+\sigma}.
\end{align*}
\end{proof}
We derive the uniform bound in $H^1$ for $(u_m)_{m\in \mathbb{N}}$ by Lemma \ref{H1lem} and Lemma \ref{H1lem1}. We have
\begin{align*}
\| u_m\|_{H^1}^2 &=\| u_m\|_{L^2}^2+\| \partial_x u_m\|_{L^2}^2 \\
&=\| u_m\|_{L^2}^2+E_m(u_m)-G_m(u_m).
\end{align*}
Applying Young's inequality, we obtain
\begin{align*}
\| u_m \|_{H^1}^2 \leq M(u_m)+E_m(u_m)+\frac{1}{2}\int_{\Omega}|u_m|^{4\sigma +2} dx +\frac{1}{2}\| \partial_x u_m\|_{L^2}^2.
\end{align*}
Hence,
\begin{align}
\| u_m\|_{H^1}^2 \leq 2M(u_m )+2E_m(u_m )+\int_{\Omega}|u_m|^{4\sigma +2} dx.
\label{H1est}
\end{align}
We introduce the following energy:
\begin{align*}
\mathcal{E}_m(u)=2M(u)+2E_m(u)+\int_{\Omega}|u|^{4\sigma +2} dx.
\end{align*}
Using Lemma \ref{H1lem}, Lemma \ref{H1lem1}, and (\ref{H1est}), we are able to conclude
\begin{align}
\frac{d}{dt}\mathcal{E}_m(u_m) 
\leq C\bigl( 1+\mathcal{E}_m(u_m)\bigr)^{3\sigma +1}.
\label{H1est1}
\end{align}
Estimates (\ref{H1est}) and (\ref{H1est1}) imply that there exists $T_0>0$ such that for all $m \in \mathbb{N}$ such that $u_m$ exists on the time interval $[-T_0,T_0]$ and
\begin{align}
M_0 \coloneqq \sup_{m\in \mathbb{N}}\| u_m \|_{C([-T_0,T_0] ;H^1)} < \infty .
\label{H1est2}
\end{align}
We note that $T_0$ depends on $\| \varphi \|_{H^1}$.\\[5pt]
%%%%%%%%%%%%%%%%%%%%%%%%%%%%%%%%
\par %段落を変える
Next, we establish the uniform $H^2$ estimates of $(u_m)_{m\in \mathbb{N}}$. 
\begin{lem}
There exists $T=T(\| \varphi \|_{H^2})>0$ which is independent of $m$ such that
\begin{align}
&u_m \in C([-T,T]; H^2(\Omega )\cap H^1_0(\Omega )),\ \forall m \in \mathbb{N},\\
& M \coloneqq \sup_{m\in \mathbb{N}} \| u_m\|_{C([-T,T];H^2)} <\infty .
\label{H2est}
\end{align}　
\end{lem}
%%%%%%%%%%%%%%%%%%%%%%%%%%%%%5
\begin{proof}
We estimate $L^2$ norm of the time derivative of $u_m$ as
\begin{align*}
\frac{d}{dt} \| \partial_t u_m\|_{L^2}^2 &=2\mathrm{Im} \biggl( i\partial_t^2 u_m, \partial_t u_m\biggr) \\
&=-2 \mathrm{Im}\biggl( \partial_t (i|u_m|^{2\sigma}\partial_x u_m), \partial_t u_m\biggr) \\
&=-2 \mathrm{Im}\biggl( i\partial_t (|u_m|^{2\sigma}) \partial_x u_m, \partial_t u_m \biggr) -2\mathrm{Re}\biggl( |u_m|^{2\sigma} \partial_x \partial_t u_m, \partial_t u_m\biggr) \\
&\leq C \| u_m\|_{L^{\infty}}^{2\sigma -1}\| \partial_x u_m\|_{L^{\infty}}\| \partial_t u_m\|_{L^2}^2 , 
\end{align*}
where in the last inequality we have used integration by parts. By Sobolev embedding and (\ref{H1est2}), we obtain
\begin{align*}
\frac{d}{dt} \| \partial_t u_m\|_{L^2}^2 \leq CM_0^{2\sigma -1}\| \partial_x u_m\|_{L^{\infty}}\| \partial_t u_m\|_{L^2}^2 .
\end{align*}
From the equation (\ref{NLS2}), we obtain
\begin{align}
\| \partial_x^2 u_m \|_{L^2} &\leq \| \partial_t u_m\|_{L^2}+\| J_mg_m(J_mu_m)\|_{L^2} \nonumber \\
&\leq \| \partial_t u_m\|_{L^2} +CM_0^{2\sigma +1}.
\label{eqest}
\end{align}
By Sobolev embedding and the conservation of charge,
\begin{align*}
\| \partial_x u_m \|_{L^{\infty}} &\leq C\| u_m\|_{H^2} \\
&\leq C(\| u_m\|_{L^2}+\| \partial_x^2 u_m \|_{L^2}) \\
&\leq C( \| \varphi \|_{L^2} +\| \partial_t u_m\|_{L^2} +CM_0^{2\sigma +1}).
\end{align*}
Applying this estimate, we deduce
\begin{align*}
\frac{d}{dt} \| \partial_t u_m\|_{L^2}^2 &\leq C(M_0) \bigl( 1+ \| \partial_t u_m\|_{L^2}\bigr) \| \partial_t u_m\|_{L^2}^2 \\
&\leq C(M_0) \bigl( 1+\| \partial_t u_m\|_{L^2}^3\bigr) .
\end{align*}
This implies that there exists $T>0$ such that $T\leq T_0$ and
\begin{align}
\sup_{m\in \mathbb{N}}\| \partial_t u_m \|_{C([-T,T] ;L^2)} < \infty .
\label{timeest}
\end{align}
From (\ref{timeest}) and (\ref{eqest}), we obtain the uniform $H^2$ estimate (\ref{H2est}).
 \end{proof}
 
%%%%%%%%%%%%%%%%%%%%%%%%%
\subsection{Convergence of the approximating sequence} 
Here we prove that $u_m$ converges in $C([-T,T];L^2(\Omega ))$ by the uniform $H^2$ estimate (\ref{H2est}). We set $I=[-T,T]$. Before proceeding to the proof, we establish the following lemma.
\begin{lem}
\label{Jm2}
Let $\varphi ,\ \psi \in C^{\infty}_{c}(\Omega )$. Then:
\begin{enumerate}[(i)]
\item $\displaystyle \| J_m\varphi -J_n\varphi \|_{L^2}\leq \left( \frac{1}{m}+\frac{1}{n}\right) \| \partial_x^2 \varphi \|_{L^2}$.
\item  $\displaystyle |(J_m \varphi -J_n \varphi , \psi )|\leq \left( \frac{1}{m}+\frac{1}{n}\right) \| \partial_x \varphi \|_{L^2}\| \partial_x \psi \|_{L^2}$.
\end{enumerate}
\begin{proof}
Let $v_m =J_m\varphi$, $v_n=J_n\varphi$. From the definition of $J_m$,
\begin{align*}
&v_m-\frac{1}{m}\partial_x^2 v_m=\varphi ,\\
&v_n-\frac{1}{n} \partial_x^2 v_n =\varphi .
\end{align*} 
Therefore,
\begin{align*}
v_m -v_n&=\frac{1}{m} \partial_x^2 v_m -\frac{1}{n}\partial_x^2 v_n \\
&=\frac{1}{m}\partial_x^2 (v_m-v_n)+\partial_x^2 v_n \left( \frac{1}{m} -\frac{1}{n}\right) .
\end{align*}
From this identity and Proposition \ref{Jm1}, the result follows.
\end{proof}
\end{lem}
We estimate $L^2$ norm of the difference $u_m-u_n$. A straightforward calculation gives us 
\begin{align*}
\frac{d}{dt}\| u_m-u_n\|_{L^2}^2 &= 2\mathrm{Im}(i\partial_t u_m-i\partial_t u_n, u_m-u_n)\\
&=-2\mathrm{Im}\Bigl( \partial_x^2 u_m-\partial_x^2 u_n,u_m-u_n ) -2\mathrm{Im}(J_mg(J_mu_m)-J_ng(J_nu_n),u_m-u_n)\\
&=-2\mathrm{Im} \biggl[ (J_mg(J_mu_m)-J_ng(J_mu_m),u_m-u_n) \\
&\quad +\Bigl( i(|J_mu_m|^{2\sigma}-|J_nu_m|^{2\sigma})J_m\partial_x u_m, J_n(u_m-u_n) \Bigr) \\
&\quad +\Bigl( i(|J_nu_m|^{2\sigma}-|J_nu_n|^{2\sigma})J_m\partial_x u_m, J_n(u_m-u_n)\Bigr) \\
&\quad +\Bigl( i|J_nu_n|^{2\sigma}(J_m\partial_x u_m-J_n\partial_x u_m),  J_n(u_m-u_n)\Bigr) \\
&\quad +\Bigl( i|J_nu_n|^{2\sigma}(J_n\partial_x u_m-J_n\partial_x u_n),  J_n(u_m-u_n)\Bigr)
\biggr] \\
&=\Rnum{1}_1+\Rnum{1}_2+\Rnum{1}_3+\Rnum{1}_4+\Rnum{1}_5.
\end{align*}
We are going to estimate each of terms $\Rnum{1}_1$, $\Rnum{1}_2$, $\Rnum{1}_3$, $\Rnum{1}_4 $ and $\Rnum{1}_5$. By Lemma \ref{Jm2} the first term is estimated as 
\begin{align*}
\Rnum{1}_1 &\leq 2\left( \frac{1}{m}+\frac{1}{n}\right) \| \partial_x g(J_mu_m)\|_{L^2}\| \partial_x (u_m-u_n)\|_{L^2}\\
&\leq C(M)\left( \frac{1}{m}+\frac{1}{n}\right) .
\end{align*}
Using an elementary inequality
\begin{align*}
||u|^{2\sigma}-|v|^{2\sigma}| \leq c\ (|u|^{2\sigma -1}+|v|^{2\sigma -1})\ |u-v|
\end{align*}
and by Lemma \ref{Jm2}, $\Rnum{1}_2$ is estimated as
\begin{align*}
\Rnum{1}_2 &\leq C(M)(\| J_mu_m\|_{L^{\infty}}^{2\sigma -1}+ \| J_nu_m\|_{L^{\infty}}^{2\sigma -1}) \| J_mu_m-J_nu_m\|_{L^2}\\
&\leq C(M)\left( \frac{1}{m}+\frac{1}{n}\right) \| \partial_x^2 u_m \|_{L^2}\\
&\leq C(M)\left( \frac{1}{m}+\frac{1}{n}\right) .
\end{align*}
%%%%%%%%%%%%%%%%%%%
A similar calculation shows 
\begin{align*}
\Rnum{1}_3 &\leq 2\| J_m\partial_x u_m \|_{L^{\infty} }
\| |J_nu_m|^{2\sigma}-|J_nu_m|^{2\sigma}\|_{L^2} \| J_n(u_m-u_n)\|_{L^2}\\
&\leq C(M)\| u_m-u_n\|_{L^2}^2 .
\end{align*}
By Lemma \ref{Jm2}, $\Rnum{1}_4$ is estimated as
\begin{align*}
\Rnum{1}_4 &\leq 2|(J_m\partial_x u_m-J_n \partial_x u_m,|J_nu_n|^{2\sigma}J_n(u_m-u_n))| \\
&\leq 2\left( \frac{1}{m}+\frac{1}{n}\right) \| \partial_x^2 u_m\|_{L^2}\| \partial_x \bigl( |J_nu_n|^{2\sigma}J_n(u_m-u_n)\bigr) \|_{L^2} \\
&\leq C(M)\left( \frac{1}{m}+\frac{1}{n}\right) .
\end{align*}
Finally, by integration by parts, $\Rnum{1}_5$ is estimated as
\begin{align*}
\Rnum{1}_5&=-2 \mathrm{Re}\left( |J_nu_n|^{2\sigma}(\partial_x J_nu_m-\partial_x J_nu_n),J_nu_m-J_nu_n\right) \\
&=\left( \partial_x (|J_nu_n|^{2\sigma}), |J_nu_m-J_nu_n|^2 \right) \\
&\leq C(M)\| u_m-u_n\|_{L^2}^2 .
\end{align*}
Gathering these estimates, we obtain
\begin{align}
\label{Cauchy1}
\frac{d}{dt} \| u_m-u_n\|_{L^2}^2 \leq C(M)\left( \frac{1}{m}+\frac{1}{n}\right) +
C(M)\| u_m-u_n\|_{L^2}^2.
\end{align}
Applying the Gronwall inequality, we obtain from (\ref{Cauchy1})
\begin{align}
\sup_{t \in I}\| u_m(t)-u_n(t)\|_{L^2}^2 \leq C(M)T\left( \frac{1}{m}+\frac{1}{n}\right) .
\end{align}
Therefore, there exists $u \in C(I; L^2(\Omega))$ such that $u_m \rightarrow u \ \mathrm{in}\ C(I;L^2(\Omega ))$. Using the elementary interpolation estimate
\begin{align*}
\| f\|_{H^s} \leq c\| f\|_{L^2}^{1-s/2}\| f\|_{H^2}^{s/2}, \ 0< s< 2
\end{align*}
and the uniform $H^2$ estimate (\ref{H2est}), we obtain $u \in C(I;H^s(\Omega)\cap H^1_0(\Omega))$ with $0\leq s <2$ such that  $u_m \rightarrow u \ \mathrm{in}\ C(I;H^s(\Omega ))$.\\

%%%%%%%%%%%%%%%%%%%%%%%%%%%
\subsection{Proof of Theorem \ref{H2lwp}}
  We shall prove that the function $u$ satisfies (\ref{NLS}) and lies in $C(I;H^2(\Omega ) \cap H^1_0(\Omega ))$. We note that $u_m$ is a solution of the integral equation
\begin{align}
u_m(t)=U(t)\varphi +i\int_{0}^{t} U(t-s)J_mg(J_mu_m(s))ds. \label{int1}
\end{align}
By Lemma \ref{Jm1} and $u_m(s) \rightarrow u(s) \ \mathrm{in}\ H^1_0(\Omega )$, we have
\begin{align*}
J_mg(J_mu_m(s))-g(u(s))&=J_m\left[ g(J_mu_m(s))-g(J_mu(s))\right]\\
&\quad +J_m\left[g(J_mu(s))-g(u(s)) \right] +J_mg(u(s))-g(u(s))\\
& \longrightarrow 0 \quad \text{as} \ \ m\rightarrow \infty
\end{align*}
in $L^2(\Omega)$ for all $s \in I$. Taking the limit in the integral equation (\ref{int1}) as $m\rightarrow \infty$, we conclude
\begin{align}
u(t) = U(t)\varphi +i\int_{0}^{t} U(t-s)g(u(s)) ds. \label{int2}
\end{align}
We set
\begin{align*}
v(t) = i\int_{0}^{t}U(t-s)g(u(s))ds. 
\end{align*}
Since $g(u) \in C(I; L^2(\Omega))$, it follows that $v\in C^1(I;L^2(\Omega))$. Since $v$ satisfies the equation 
\begin{align}
i\partial_t v+\partial_x^2 v+g(u)=0, 
\end{align}
it follows that $\partial_x^2 v \in C(I; L^2(\Omega))$. Therefore, $u \in C(I; H^2(\Omega))$ follows from the integral equation (\ref{int2}). The uniqueness and continuous dependence is verified by the same argument as in \cite{Ambrose}.
%%%%%%%%%%%%%%%%%%%%%%%%%%%%%%%%%%%
%%%%%%%%%%%%%%%%%%%%%%%%%%%%%%%%%%%
\section{Proof of Theorem {\ref{H3/2}}}
\label{SH3/2}
For the proof of Theorem {\ref{H3/2}}, the following lemma is essential (see, for example, \cite{OO}). 
\begin{lem}
\label{Lp}
Let $p \in [2,\infty )$. For any $u\in H^{1/2}(\Omega )$, 
\begin{align}
\| u \|_{L^p} \leq C\sqrt{p}\| u\|_{H^{1/2}},
\end{align}
where $C$ is independent of $p$.
\end{lem}
We set 
\begin{align*}
M=\max \{ \| u\|_{L^{\infty}((-T,T);H^{3/2})}, 
 \| v\|_{L^{\infty}((-T,T);H^{3/2})}\} .
\end{align*}
Using integration by parts and H\"{o}lder's inequality, we obtain
\begin{align*}
\frac{d}{dt} \| u-v\|_{L^2}^2&=2\mathrm{Im}(i\partial_t u-i\partial_t v, u-v) \\[5pt] 
&=-2\mathrm{Re}\Bigl( (|u|^{2\sigma}-|v|^{2\sigma})\partial_x u, u-v \Bigr) -2\mathrm{Re}\Bigl( |v|^{2\sigma}(\partial_x u-\partial_x v), u-v \Bigr) \\[5pt]
&\leq C(M) \int_{\Omega} (|\partial_x u|+|\partial_x v|) |u-v|^2 dx \\[5pt]
&\leq C(M) (\| \partial_x u\|_{L^p}+\| \partial_x v\|_{L^p})\| u-v\|_{L^{2p'}}^2 .
\end{align*}
By H\"{o}lder's inequality
\begin{align*}
\| u-v \|_{L^{2p'}} \leq \| u-v\|_{L^2}^{1/p'}\| u-v\|_{L^{\infty}}^{1-1/p'},
\end{align*}
Sobolev embedding and Lemma \ref{Lp}, we obtain
\begin{align}
\frac{d}{dt} \| u-v\|_{L^2}^2 &\leq C(M)\sqrt{p}(\| u\|_{H^{3/2}}+\| v\|_{H^{3/2}}) \| u-v\|_{L^2}^{2(1-1/p)} \notag \\
&\leq C(M)\sqrt{p}\| u-v\|_{L^2}^{2(1-1/p)} , \label{Yud1}
\end{align}
where $C(M)$ is still independent of $p$. It follows from (\ref{Yud1})
\begin{align*}
\frac{d}{dt} \| u-v\|_{L^2}^{2/p} \leq \frac{C(M)}{\sqrt{p}}.
\end{align*}
By integration in time, we deduce
\begin{align}
\| u(t)-v(t)\|_{L^2}^2 \leq \left(  \frac{C(M)T}{\sqrt{p}}\right) ^p \label{Yud2} 
\end{align}
for all $t \in (-T,T)$. Since the RHS of (\ref{Yud2}) goes to $0$ as $p\rightarrow \infty$, we conclude $u=v$. 
%%%%%%%%%%%%%%%%%%%%%%%%%%%%%%%%%%%
%%%%%%%%%%%%%%%%%%%%%%%%%%%%%%%%%%%
\section{Well-posedness in $H^1$}
\label{SH1}
 In this section, we consider $H^1$ solutions of (\ref{NLS}). Specifically, we shall prove Theorem \ref{H1lwp} and Theorem \ref{H1gwp}.
\subsection{The gauge transformation}
 Let $u$ is a solution of (\ref{NLS}). We derive a differential equation of $\partial_x u$. To that end, we follow an idea in \cite{Ozawa}. We define the differential operator
 \begin{align*}
 L=i\partial_t +\partial_x^2 .
 \end{align*}
 A direct calculation shows
 \begin{align}
 e^{\Lambda}L(e^{-\Lambda }\partial_x u)=L\partial_x u+\Bigl((\partial_x \Lambda )^2-L\Lambda \Bigr)\partial_x u -2\partial_x \Lambda \partial_x^2 u ,
 \label{cal1}
 \end{align}
 where $\Lambda$ is a function. We note
 \begin{align}
 L\partial_x u=\partial_x Lu =-i|u|^{2\sigma}\partial_x^2 u -i\partial_x (|u|^{2\sigma})\partial_x u .
 \label{Ldelu}
 \end{align}
 Let $\Omega =(a,b)$ with $-\infty \leq a<b\leq \infty$. To absorb the worst term $ -i|u|^{2\sigma}\partial_x^2 u$ by means of $ -2\partial_x \Lambda \partial_x^2 u$ on the RHS of (\ref{cal1}), we set
 \begin{align}
 \Lambda = -\frac {i}{2} \int_{a}^{x} |u(t,y)|^{2\sigma} dy .
 \end{align}
 We compute $i \partial_t \Lambda$ as
 \begin{align*}
 i\partial_t \Lambda &=\frac{1}{2}\int_{a}^{x} 2\sigma |u|^{2(\sigma -1)}\mathrm{Re}(\overline{u}\partial_t u) dy \\
 &=\sigma \int_{a}^{x} |u|^{2(\sigma -1)}\mathrm{Im}(\overline{u}(
-\partial_x^2 u-i|u|^{2\sigma}\partial_x u)) dy \\
&=-\sigma \mathrm{Im} (|u|^{2(\sigma -1)}\overline{u}\partial_x u)+\sigma \mathrm{Im} \left[ \int_{a}^{x} \partial_x (|u|^{2(\sigma -1)}\overline{u}) \partial_x u dy \right] -\sigma \int_{a}^{x} |u|^{2(2\sigma -1)} \mathrm{Re}(\overline{u}\partial_x u) dy \\
&= -\sigma \mathrm{Im} (|u|^{2(\sigma -1)}\overline{u}\partial_x u)+\sigma \mathrm{Im} \left[ \int_{a}^{x} \partial_x (|u|^{2(\sigma -1)}\overline{u}) \partial_x u dy \right] -\frac{1}{4} |u|^{4\sigma} .
 \end{align*}
 Therefore,
 \begin{align}
 (\partial_x \Lambda )^2 -L\Lambda =\sigma \mathrm{Im} (|u|^{2(\sigma -1)}\overline{u}\partial_x u)-\sigma \mathrm{Im} \left[ \int_{a}^{x} \partial_x (|u|^{2(\sigma -1)}\overline{u}) \partial_x u dy \right] +\frac{i}{2}\partial_x (|u|^{2\sigma}).
 \label{cal2}
 \end{align}
 Collecting (\ref{cal1})-(\ref{cal2}), we obtain
 \begin{align}
 e^{\Lambda}L(e^{-\Lambda }\partial_x u) =Q_1(u)+Q_2(u),
 \end{align}
 where
 \begin{align*}
&Q_1(u)=-\frac{i}{2}\partial_x( |u|^{2\sigma})\partial_x u+\sigma \mathrm{Im}(|u|^{2(\sigma -1)}\overline{u}\partial_x u)\partial_x u,\\
&Q_2(u)=-\sigma \int_{a}^{x} \mathrm{Im} \Bigl( \partial_x (|u|^{2\sigma -2}\overline{u})\partial_x u\Bigr) dy \partial_x u.
 \end{align*}
 To prove Theorem \ref{H1lwp}, we approximate $\varphi \in H^1_0(\Omega )$ by a sequence $(\varphi_n )_{n\in \mathbb{N}}$ such that $\varphi_n \in H^2(\Omega )\cap H^1_0 (\Omega )$ and $\varphi_n \rightarrow \varphi \ \mathrm{in}\ H^1_0(\Omega )$. By Theorem \ref{H2lwp}, 
(\ref{NLS}) has a solution
 \begin{align*}
 u_n \in C([-T_n,T_n]; H^2(\Omega )\cap H^1_0(\Omega ))
 \end{align*}
with $u_n(0)=\varphi_n $. We set $I_n=[-T_n,T_n]$. Since the formal calculation above is justified with $u$ replaced by $u_n$, we obtain 
 \begin{align}
u_n(t) &= U(t) \varphi_n +iG(g(u_n(t))) ,
\label{un1}\\
e^{-\Lambda_n (t)}\partial_x u_n(t) &=U(t)(e^{-\Lambda_n (0)}\partial_x \varphi_n) +i G\Bigl( e^{-\Lambda_n (t)}\bigl( Q_1(u_n(t))+Q_2(u_n(t))\bigr) \Bigr) 
\label{un2} 
\end{align}
for all $t \in I_n$, where
\begin{align*}
&\Lambda_n =- \frac{i}{2}\int_{a}^{x} |u_n(t,y)|^{2\sigma}dy, \\
&G(v)(t)=\int_{0}^{t} U(t-s)v(s)ds .
\end{align*}
%%%%%%%%%%%%%%
%%%%%%%%%%%%%%
\subsection{The uniform estimate in $H^1$}
 To derive the uniform estimate in $H^1$ of the approximate solutions $(u_n)_{n\in \mathbb{N}}$, we use the following Strichartz estimate. The proofs can be found in \cite{Cazenave}.
 \begin{prop}
 \label{Strichartz}
%  \leavevmode \par
Assume $\Omega$ is an unbounded interval, then the following properties hold:
\begin{enumerate}[(i)]
\item  For any $(q,r)$ with $0 \leq 2/q=1/2-1/r \leq 1/2$,
\begin{align*}
\|  U (\cdot ) \varphi \|_{L^q (\mathbb{R}; L^r(\Omega ))} \leq C\ \| \varphi \|_{L^2(\Omega )}.
\end{align*}
\item For any $(q_j,r_j)$ with $0 \leq 2/q_j=1/2-1/r_j \leq 1/2,\ j=1,2$ for any interval 
$I \subset \mathbb{R}$ with $0 \in \overline{I}$，
\begin{align*}
\| G(v)\|_{L^{q_1}(I; L^{r_1})} \leq C\| v\|_{L^{q_2'}(I; L^{r_2'})},
\end{align*}
where the constant $C$ is independent of $I$.
\end{enumerate}
 \end{prop}
 Before proceeding the proof, we introduce function spaces. For a time interval $I$, we define $\mathcal{X}_0(I)$ and $\mathcal{X}(I)$ the function spaces by
\begin{align*}
&\mathcal{X}_0(I)= \bigcap_{0 \leq 2/q=1/2-1/r \leq 1/2} L^q(I ; L^r(\Omega) ),\\
&\mathcal{X}(I)= \bigcap_{0 \leq 2/q=1/2-1/r \leq 1/2} L^q(I ; W^{1,r}(\Omega )) ,
\end{align*} 
with norms
\begin{align*}
&\| u\|_{\mathcal{X}_0(I)} = \sup_{0 \leq 2/q=1/2-1/r \leq 1/2} \| u\|_{L^q(I ; L^r )}, \\
&\| u\|_{\mathcal{X}(I)}=\| u\|_{\mathcal{X}_0(I)}+\| \partial_x u\|_{\mathcal{X}_0(I)} .
\end{align*}
Applying Proposition \ref{Strichartz} to (\ref{un1}) and (\ref{un2}), and using Sobolev embedding and H\"{o}lder's inequality, we obtain
\begin{align}
\| u_n \|_{\mathcal{X}_0 (I_n)} &\leq C \| \varphi_n \|_{L^2} +C\| |u_n|^{2\sigma} \partial_x u_n \|_{L^1(I_n;L^2)} \nonumber \\
&\leq C\| \varphi_n \|_{L^2} +CT_n \| u_n \|^{2\sigma +1}_{\mathcal{X}(I_n)}
, \label{St1}\\[10pt]
%%%%%%%%%%%%%%%%%%
\| \partial_x u_n \|_{\mathcal{X}_0 (I_n)} &=\| e^{-\Lambda_n} \partial_x u_n \|_{\mathcal{X}_0 (I_n)} \nonumber \\
&\leq C\| e^{-\Lambda_n (0)}\partial_x \varphi_n \|_{L^2} 
+C\Bigl( \| e^{-\Lambda_n}Q_1(u_n)\|_{L^{\frac{3}{4}}(I_n;L^1) }
+\| e^{-\Lambda_n}Q_2(u_n) \|_{L^1(I_n;L^2)}\Bigr) \nonumber \\
&\leq C\| \partial_x \varphi_n \|_{L^2} +C (T_n^{\frac{3}{4}}+T_n)\| u_n\|^{2\sigma +1}_{\mathcal{X}(I_n)}, \label{St2}
\end{align}
where the constant $C$ is independent of $n$. Collecting (\ref{St1}) and (\ref{St2}), we obtain
\begin{align}
\| u_n\|_{\mathcal{X}(I_n)} 
\leq C M+C (T_n+T_n^{\frac{3}{4}})\| u_n\|^{2\sigma +1}_{\mathcal{X}(I_n)} ,
\label{St3}
\end{align}
where $M$ is given by
\begin{align*}
M\coloneqq \sup_{n \in \mathbb{N}} \| \varphi_n \|_{H^1}.
\end{align*}
We conclude from (\ref{St3}) easily that there exists $T>0$ such that for all $m \in \mathbb{N}$ such that $u_m$ exists on the time interval $I \coloneqq [-T,T]$ and
\begin{align}
\sup_{m\in \mathbb{N}}\| u_m \|_{\mathcal{X}(I)} \leq 2CM. \label{uniform}
\end{align}
%%%%%%%%%%%%%%%%%%
\subsection{Proof of Theorem \ref{H1lwp}}
 Firstly, we prove that $u_m$ converges in $C(I;L^2(\Omega ))$ by the uniform estimate (\ref{uniform}). A straightforward calculation shows
 \begin{align}
\frac{d}{dt} \| u_n-u_m\|_{L^2}^2 &=2\mathrm{Im}(i\partial_t u_n-i\partial_t u_m, u_n-u_m)\notag \\[5pt]
&=-2\mathrm{Im} (\partial^2_x u_n -\partial^2_x u_m, u_n-u_m)-2\mathrm{Re}(|u_n|^{2\sigma}\partial_x u_n-|u_m|^{2\sigma}\partial_x u_m, u_n-u_m) \notag \\[5pt]
&=-2\mathrm{Re}\Bigl( (|u_n|^{2\sigma}-|u_m|^{2\sigma})\partial_x u_n, u_n-u_m\Bigr) -2\mathrm{Re} \Bigl( |u_m|^{2\sigma} (\partial_x u_n-\partial_x u_m), u_n-u_m\Bigr) \notag \\[5pt]
&\leq C \Bigl( \| u_n\|_{L^{\infty}}^{2\sigma -1}+\| u_m\|_{L^{\infty}}^{2\sigma -1}\Bigr) \Bigl( \| \partial_x u_n\|_{L^{\infty}}+\| \partial_x u_m\|_{L^{\infty}}\Bigr) \| u_n-u_m\|_{L^2}^2 \notag \\[5pt]
&\leq C(M) \Bigl( \| \partial_x u_n\|_{L^{\infty}}+\| \partial_x u_m\|_{L^{\infty}}\Bigr) \| u_n-u_m\|_{L^2}^2.
\label{Cauchy2}
\end{align}
Applying the Gronwall inequality, we obtain from (\ref{Cauchy2})
\begin{align*}
\sup_{t\in I} \| u_n(t)-u_m(t) \|_{L^2}^2 \leq \| \varphi_n -\varphi_m \|_{L^2}^2 \mathrm{exp}(C(M)T^{\frac{1}{4}}) .
\end{align*}
This implies that there exists $u \in C(I; L^2(\Omega))$ such that 
\begin{align}
u_m \rightarrow u \ \mathrm{in}\ C(I;L^2(\Omega )).  \label{L2}
\end{align}
By the interpolation inequality, 
\begin{align}
u_n \rightarrow u \quad \mathrm{in} \ C(I; L^r(\Omega )) 
\label{Lr}
\end{align}
for any $r$ with $2\leq r < \infty$. Since $W^{1,r}(\Omega )$ is reflexive if $(q,r)\text{ satisfies} \ 0 \leq 2/q=1/2-1/r < 1/2$, we obtain from (\ref{uniform}) and (\ref{Lr}) 
\begin{align}
\| u \|_{L^q(I; W^{1,r})} \leq \liminf_{n \rightarrow \infty} \| u_n \|_{L^q(I; W^{1,r})}  \leq 2CM
\label{Lr2}
\end{align}
for any $r$ with $2\leq r < \infty$. Since the constant on the RHS of (\ref{Lr2}) is independent of $(q,r)$, taking the limit as $r\rightarrow \infty$, we conclude
\begin{align*}
\| u \|_{L^4(I; W^{1,\infty})} \leq 2CM .
\end{align*}
Therefore, $u \in \mathcal{X}(I)$. We see that $u$ is a solution of (\ref{NLS}) in the distribution sense. We note that the approximate solution $u_m$ conserves the charge and energy. By (\ref{L2}), we obtain $M(u(t))=M(\varphi )$ for all $t \in I$. To show $u$ conserves the energy, we need the following lemma.
\begin{lem}
\label{polem}
Let $\sigma >0$. For every $M>0$, there exists $C(M)>0$, we have　
\begin{align}
| G(u)-G(v)| \leq C(M)\| u-v \|_{L^2}
\label{potential}
\end{align}
for all $u, v \in H^1_0(\Omega )$ such that $\| u\|_{H^1}, \| v\|_{H^1}\leq M$.
\end{lem}
\begin{proof}
Since $G'(u)=g(u)$, we obtain
\begin{align*}
G(u)-G(v) &=\int_{0}^{1} \frac{d}{ds} G(su+(1-s)v)ds \\
&=\int_{0}^{1} 2\mathrm{Re}\biggl( g(su+(1-s)v), u-v \biggr) ds .
\end{align*}
From this identity and Sobolev embedding, the inequality (\ref{potential}) follows.
\end{proof}
By (\ref{uniform}) and (\ref{L2}), we note that $u_m(t) \rightharpoonup u(t) \ \mathrm{in} \ H^1_0(\Omega )$. By the weak lower semicontinuity of the norm, (\ref{L2}) and 
Lemma \ref{polem}, we obtain
\begin{align}
E(u(t)) &\leq \liminf_{m \rightarrow \infty}\bigl( \| \partial_x u_m(t)\|_{L^2}^2+G(u_m(t)) \bigr) \notag \\
&= \liminf_{m \rightarrow \infty} E(u_m(t)) =E(\varphi )
\label{energyineq}
\end{align}
for all $t \in I$.\\
\par
Next, we prove that $u$ is the unique solution of (\ref{NLS}). Suppose that $v\in L^{\infty}(I;H^1_0(\Omega )) \cap L^4(I; W^{1,\infty }(\Omega ))$ is also a solution of (NLS). We set
\begin{align*}
M=\max \{ \| u \|_{L^{\infty}(I;H^1_0)} +\| u\|_{L^4(I; W^{1,\infty})}, \| v \|_{L^{\infty}(I;H^1_0)} +\| v\|_{L^4(I; W^{1,\infty})} \}
\end{align*}
By the same calculation as (\ref{Cauchy2}), we obtain
\begin{align}
\frac{d}{dt} \| u-v\|_{L^2}^2 
\leq C(M) \Bigl( \| \partial_x u\|_{L^{\infty}}+\| \partial_x v\|_{L^{\infty}}\Bigr)
 \| u-v \|_{L^2}^2 .
 \label{unique}
\end{align}
Applying the Gronwall inequality to (\ref{unique}), we conclude that $u=v$ on $I$. By uniqueness and (\ref{energyineq}), we deduce easily that
\begin{align}
E(u(t))=E(\varphi )
\end{align}
for all $t\in I$ and that $u \in C(I;H^1_0(\Omega ))$. \\
\par
Finally, we prove the continuous dependence. Suppose that $\varphi_n \rightarrow \varphi \ \mathrm{in}\ H^1_0(\Omega )$  and let $u_n$ be a solution of (\ref{NLS}) with $u_n(0)=\varphi $. By the same calculation as (\ref{Cauchy2}), we deduce 
\begin{align}
u_n \rightarrow u \ \mathrm{in}\ C(I;L^2(\Omega )). 
\end{align}
By the conservation of charge and energy, and Lemma \ref{polem}, we obtain 
\begin{align}
\| u_n(t)\|_{H^1} \rightarrow \| u(t)\|_{H^1} 
\end{align}
uniformly on $I$. Therefore, we conclude that $u_n \rightarrow u \ \mathrm{in}\ C(I; H^1_0(\Omega ))$.
%%%%%%%%%%%%%%%%%%%%%%%%%%
%%%%%%%%%%%%%%%%%%%%%%%%%%
\subsection{Proof of Theorem \ref{H1gwp}}
 We only prove Theorem \ref{H1gwp} in the case $\sigma >1$. For the proof when $\sigma =1$, see \cite{HO1} or \cite{Wu}. We assume that $u \in C([-T,T]; H^1_0(\Omega ))$  is a solution of (\ref{NLS}). By the conservation of energy and Sobolev embedding, we obtain
 \begin{align*}
 \| \partial_x u\|_{L^2}^2 &= E(\varphi )-G(u) \\[5pt]
 &\leq E(\varphi ) +\frac{1}{\sigma +1} \| u\|_{L^{4\sigma +2}}^{2\sigma +1}\| \partial_x u\|_{L^2} \\
 &\leq E(\varphi ) +\frac{c}{\sigma +1} \| u \|_{H^1}^{2\sigma +2} .
 \end{align*}
 By the conservation of charge, we obtain
 \begin{align}
 f_{\sigma }(\| u\|_{H^1}) \coloneqq \| u\|_{H^1}^2-\frac{c}{\sigma +1} \| u\|_{H^1}^{2\sigma +2} \leq M(\varphi )+E(\varphi ).
 \label{apriori1}
 \end{align}
 We note that $f_{\sigma}$ has an unique local maximum at $\delta >0$, where $\delta$ is 
given by $\delta^{2\sigma}=c^{-1}$. If $\varphi \in H^1_0(\Omega)$ satisfies
\begin{align*}
M(\varphi )+E(\varphi )<f_{\sigma}(\delta )\ \text{and}\  \| \varphi \|_{H^1} <\delta ,
\end{align*}
then by (\ref{apriori1})
\begin{align}
f_{\sigma}(\| u(t)\|_{H^1}) \leq M(\varphi )+E(\varphi )<f_{\sigma}(\delta )
\label{apriori2}
\end{align}
for all $t \in [-T,T]$. Since $\| \varphi \|_{H^1} <\delta $ and $\| u(t)\|_{H^1}$ is continuous, we deduce
\begin{align}
\sup_{t\in [-T,T]}\| u(t) \|_{H^1} < \delta .
\label{apriori}
\end{align}
From the a priori estimate (\ref{apriori}) and Theorem \ref{H1lwp}, the result follows.

%%%%%%%%%%%%%%%%%%%%%%%%%%%%%%%%%
%%%%%%%%%%%%%%%%%%%%%%%%%%%%%%%%%
\section{Proof of Theorem \ref{H1global}}
\label{SH12}
We recall the following approximate problem in Section \ref{SH2}:
 \begin{align}
 \begin{cases}
 i \partial_{t}u_m + \partial_{x}^2 u_m +J_mg(J_mu_m)= 0, \\
 u_m(0)=\varphi .
 \end{cases}
 \label{NLS3}
 \end{align}
We note that $\partial_x^2$ is self-adjoint in $H^{-1}(\Omega )$ with domain $H^1_0(\Omega )$. Let $\varphi \in H^1_0(\Omega )$ be given. It is easily verified that there exists a sequence $(u_m)_{m\in \mathbb{N}}$ of functions of $C((-T_m,T_m); H^1_0(\Omega ))$ such that satisfies (\ref{NLS3}) and
\begin{align}
M(u_m (t))=M(\varphi ) \quad \text{and} \quad E_m(u_m(t))=E_m(\varphi )
\label{mcons}
\end{align} 
for all $t \in (-T_m,T_m)$, where $E_m$ is defined as (\ref{menergy}). We use the conservation of energy in order to obtain the uniform $H^1$ estimates of $(u_m)_{m\in \mathbb{N}}$. We have
\begin{align*}
\| \partial_x u_m \|_{L^2}^2 &= E_m(\varphi )-G_m(u_m) \\[5pt]
 &\leq E_m (\varphi ) +\frac{1}{\sigma +1} \| J_mu_m \|_{L^{4\sigma +2}}^{2\sigma +1}\| \partial_x J_mu_m \|_{L^2} .
\end{align*}
By using Gagliardo-Nirenberg's inequality 
\begin{align*}
\| f\|_{L^{4\sigma +2}}^{2\sigma +1} \leq C\| f\|_{L^2}^{\sigma +1}\| \partial_x f\|_{L^2}^{\sigma}
\end{align*}
and Proposition \ref{Jm1}, we obtain
\begin{align}
\| \partial_x u_m \|_{L^2}^2 &\leq E_m(\varphi )+\frac{C}{\sigma +1}
\| u_m\|_{L^2}^{\sigma +1}\| \partial_x  u_m\|_{L^2}^{\sigma +1} \notag \\
&=E_m(\varphi )+\frac{C}{\sigma +1} \| \varphi \|_{L^2}^{\sigma +1}\| \partial_x u_m\|_{L^2}^{\sigma +1} ,
\label{H1est3}
\end{align}
where in the last equality we have used the conservation of charge. Since $\sigma +1<2$, applying 
Young's inequality to (\ref{H1est3}), we deduce the following estimate
\begin{align*}
\| \partial_x u_m(t)\|_{L^2}^2 \leq C(\| \varphi \|_{H^1}) 
\end{align*}
for all $t \in (-T_m,T_m)$. This implies that $T_m= \infty$ for every $m\in \mathbb{N}$ and
\begin{align}
M \coloneqq \sup_{m \in \mathbb{N}} \| u_m\|_{C(\mathbb{R};H^1_0)} <\infty .
\label{H1est4}
\end{align}
By the equation (\ref{NLS3}) and the estimate $\| g_m(u_m(t))\|_{L^2} \leq C(M)$ for all $t\in \mathbb{R}$, we obtain
\begin{align}
\sup_{m\in \mathbb{N}}\| \partial_t u_m \|_{C(\mathbb{R};H^{-1})} \leq C(M) .
\label{H1est5} 
\end{align}
By applying (\ref{H1est4}), (\ref{H1est5}), and the abstract version of the 
Ascoli-Arzel\`{a} theorem, we deduce that
\begin{align*}
u \in L^{\infty}(\mathbb{R}; H^1_0(\Omega )) \cap W^{1,\infty}(\mathbb{R};H^{-1}(\Omega ))
\end{align*}
and that there exists a subsequence, which we still denote by $(u_m)_{m\in \mathbb{N}}$, such that
\begin{align}
u_m(t) \rightharpoonup u(t) \ \mathrm{in} \ H^1_0(\Omega )
\label{H1weak}
\end{align}
for all $t \in \mathbb{R}$. To prove that $u$ is a weak solution of (\ref{NLS}), we need the following lemma. 
%%%%%%%%%%%%%%%%%%%%%%%%%
\begin{lem}
\label{glem}
For all $t \in \mathbb{R}$, $g_m(u_m(t)) \rightharpoonup g(u(t))$ in $L^2(\Omega )$ .
\end{lem}
\begin{proof}
Let $\omega \in C^{\infty}_{c}(\Omega )$ and let $B=\mathrm{supp} \ \omega$. We write
\begin{align*}
\Bigl( g_m(u_m)-g(u) ,\omega \Bigr) &=\Bigl( J_mg(J_mu_m)-g(J_mu_m) , \omega  \Bigr) \\[5pt]
&\quad +\Bigl( i|Ju_m|^{2\sigma}\partial_x J_mu_m-i|u_m|^{2\sigma}\partial_x J_mu_m, \omega \Bigr) \\[5pt]
&\quad +\Bigl( i|u_m|^{2\sigma}\partial_x J_mu_m-i|u|^{2\sigma}\partial_x J_mu_m, \omega \Bigr) \\[5pt]
&\quad +\Bigl( i|u|^{2\sigma}\partial_x J_mu_m-i|u|^{2\sigma}\partial_x u_m, \omega \Bigr) \\[5pt]
&\quad +\Bigl( i|u|^{2\sigma}\partial_x u_m-i|u|^{2\sigma}\partial_x u, \omega \Bigr) \\[5pt]
&=\Rnum{2}_1+\Rnum{2}_2+\Rnum{2}_3+\Rnum{2}_4+\Rnum{2}_5. 
\end{align*}
Since $g(J_mu_m)$ is bounded in $L^2(\Omega )$ due to (\ref{H1est4}), $\Rnum{2}_1 \rightarrow 0$ by Proposition \ref{Jm1}. In the case of $1/2\leq \sigma <1$, we estimate $\Rnum{2}_2$ as
\begin{align*}
|\Rnum{2}_2| &\leq \| \omega \|_{L^{\infty}} \| \partial_x J_mu_m \|_{L^2}\| \ |J_mu_m|^{2\sigma} -|u_m|^{2\sigma} \|_{L^2(B)} \\
&\leq C(M)\| J_mu_m-u_m\|_{L^2(B)} .
\end{align*}
Since $u_m$ is bounded in $H^1_0(\Omega )$, it follows $J_mu_m -u_m 
\rightharpoonup 0$ in $H^1_0(\Omega )$, hence  $J_mu_m -u_m \rightarrow 0$ in $L^2(B )$ by the Rellich-Kondrachov theorem. Therefore, $\Rnum{2}_2 \rightarrow 0$. In the case of $0< \sigma <1/2$, we estimate $\Rnum{2}_2$ as
\begin{align*}
|\Rnum{2}_2| &\leq C(M)\| \ |J_mu_m|^{2\sigma} -|u_m|^{2\sigma} \|_{L^2(B)} \\
&\leq C(M)\| J_mu_m-u_m\|_{L^{4\sigma}(B)}^{2\sigma} \\
&\leq C(M)|B|^{\frac{1-2\sigma}{2}}\| J_mu_m-u_m\|_{L^2(B)}^{2\sigma}\\
& \longrightarrow 0 \quad \text{as} \ \ m\rightarrow \infty .
\end{align*}
Here, we used an elementary inequality 
\begin{align*}
||u|^{2\sigma}-|v|^{2\sigma}|\leq |u-v|^{2\sigma}
\end{align*}
in the second inequality.
Similarly, we can show 
$\Rnum{2}_3,\  \Rnum{2}_4 \rightarrow 0$. Since $\partial_x u_m \rightharpoonup \partial_x u$ in $L^2(\Omega )$, we deduce $\Rnum{2}_5 \rightarrow 0$.
\end{proof}
It follows from (\ref{H1weak}) and Lemma \ref{glem} that $u$ is a solution of (\ref{NLS}) in the distribution sense. Taking the $H^{-1}$-$H^1_0$ duality product of the equation (\ref{NLS}), we deduce 
\begin{align}
\frac{d}{dt} \| u(t) \|_{L^2}^2 =0 
\end{align}
for all $t \in \mathbb{R}$, and so
\begin{align}
M(u(t)) =M(\varphi ) .
\label{charge}
\end{align}
By (\ref{mcons}), (\ref{charge}) and (\ref{H1weak}), we deduce
\begin{align}
u_m \rightarrow u \ \mathrm{in} \ C_{\mathrm{loc}}(\mathbb{R};L^2(\Omega )) .
\label{CL2}
\end{align}
It follows from (\ref{mcons}), (\ref{H1weak}), (\ref{CL2}) and Lemma \ref{polem} that
\begin{align}
E(u(t)) \leq E(\varphi )
\end{align}
for all $t \in \mathbb{R}$. This completes the proof.\\[30pt]
\noindent
{\normalsize\bfseries\boldmath Acknowledgments}\\[15pt]
\indent We would like to thank the referees for helpful comments concerning Theorem \ref{H1global}.
%% The Appendices part is started with the command \appendix;
%% appendix sections are then done as normal sections
%% \appendix

%% \section{}
%% \label{}

%% If you have bibdatabase file and want bibtex to generate the
%% bibitems, please use
%%
%%  \bibliographystyle{elsarticle-num} 
%%  \bibliography{<your bibdatabase>}

%% else use the following coding to input the bibitems directly in the
%% TeX file.

\end{document}